\documentclass{article}

\newcommand{\vekk}[1]{}
\input{def.tex}

\begin{document}

\title{Improper posteriors are not improper}
\author{Gunnar Taraldsen, Jarle Tufto, and Bo H. Lindqvist \\
  \\
Department of Mathematical Sciences \\ Norwegian University of Science and Technology
}

\maketitle

\begin{abstract}
  In 1933 Kolmogorov constructed a general theory that
  defines the modern concept of conditional expectation. 
  In 1955 Rényi fomulated a new axiomatic theory for probability motivated
  by the need to include unbounded measures.
  We introduce a general concept of conditional expectation in Rényi spaces.  
  In this theory improper priors are allowed, and the resulting posterior can also be improper.
  
  In 1965 Lindley published his classic text on Bayesian statistics using the theory of Rényi,
  but retracted this idea in 1973 due to the appearance of marginalization paradoxes presented
  by Dawid, Stone, and Zidek.
  The paradoxes are investigated, and the seemingly conflicting results are explained.
  The theory of Rényi can hence be used as an axiomatic basis for statistics that
  allows use of unbounded priors.
  
  {\it Keywords:} {\bf Haldane's prior; Poisson intensity; Marginalization paradox; Measure theory;
  conditional probability space; axioms for statistics; conditioning on a sigma field; improper prior}
\end{abstract}

\tableofcontents

\newpage

\section{Introduction}

An often voiced criticism of the use of improper priors in Bayesian inference 
is that such priors sometimes don't lead to a proper posterior distribution.
This can happen when the marginal prior distribution of the data is not
$\sigma$-finite \citep{TaraldsenLindqvist10ImproperPriors},
as sometimes encountered in applied settings with sparse data \citep[e.g.][Appendix S4]{Druilhet2016,Tufto2012butterflies}.

As a simple motivating example, suppose that we observe a homogeneous Poisson process,
and that we start with a non-informative scale prior
$\pi (\lambda)=1/\lambda$ on the Poisson intensity $\lambda$.
The distribution of the number $X_1$ of events in the interval $(0,t_1]$
is then not $\sigma$-finite since
$P(X_1 = 0) =\int_0^\infty P(X_1=0 | \lambda) \pi(\lambda)\, d \lambda$ is infinite.
If we observe $X_1=0$ and formally multiply
the prior by the likelihood we obtain an improper posterior
$\pi (\lambda | X_1=0) \propto e^{-\lambda t_1}/\lambda$.
This distribution for $\lambda$ is different from the initial prior, 
and we claim that this is a correct way of incorporating the
information given by $X_1=0$.

A related example is the Beta posterior density 
for the success probability $p$ given by 
$\pi (p \st \alpha) = p^{\alpha - 1} (1-p)^{\beta - 1}$
for a Bernoulli sequence with $\alpha$ successes out of $\alpha + \beta$ trials.
This corresponds to the \citet{Haldane32prior} improper prior
$\pi (p) = p^{- 1} (1-p)^{- 1}$, 
and the posterior is improper if
$\alpha$ or $\beta$ is zero.
In all cases, however, the observation of the number of successes $\alpha$ results in a corresponding updating of the uncertainty associated with $p$.
This is given by the possibly improper posterior.

Unfortunately, 
even people accepting the use of improper priors reject this form of inference,
on the ground that the posterior is not a probability distribution,
and a mathematical theory is lacking for this
\citep{RobertChopinRousseau09jeffreys}.
This is understandable, and we agree initially with this point of view. 
We will demonstrate, however, 
that the above forms of, so far,
formal inference can be made consistent with
the axiomatic system of Rényi which allows improper laws.
We propose and claim that the mathematical theory developed in the
following gives a rigorous foundation for inference
with unbounded laws.

The most familiar example of an unbounded law $\pr_{\Theta}$ is
the uniform law on the real line $\Omega_\Theta = \RealN$.
Following \citet{RENYI} and \citet{TaraldsenLindqvist16renyi} 
the uniform law is identified with
the countable collection of uniform laws
$\Pr_\Theta (\cdot | B_n)$ on each interval $B_n = [-n, n]$.
This gives then also the interpretation of $\pr_{\Theta}$:
Given that $\Theta \in B_n$ the law is the uniform probability 
distribution on $B_n$.
The family $\cB = \{B_n \st n \in \NatN\}$ is a {\em bunch} and
the family $\{\pr_\Theta (\cdot \st B) \st B \in \cB\}$ defines a Rényi space.
The improper laws for the intensity $\lambda$ and the success probability $p$  
in the initial examples are interpreted similarly.
The concept of a Rényi space and other elements from measure theory
are summarized in Appendix~\ref{sDef}.

The aim of this paper is to present a theory of statistics that
allows improper laws both as priors and posteriors.
This extends the results of 
\citet{TaraldsenLindqvist10ImproperPriors,TaraldsenLindqvist16renyi},
and provides stronger links to the results on improper
laws presented by \citet{HARTIGAN} and \citet{BiocheDruilhet16convergence}.
The main mathematical result is Theorem~\ref{theo1} which
proves existence and uniqueness of conditional expectation
on Rényi spaces.
The existence and uniqueness proof relies on the Radon-Nikodym theorem
and a generalization of the Rényi structure theorem.
The theory of conditional expectation has been most important
for the development of measure theory, probability and statistics
based on Kolmogorov's concept of a probability space.
The generalization of this to the setting of Rényi spaces can hence
be expected to be important for future developments
in mathematics and related fields. 

Within this framework we reach the view that improper posteriors,
just as improper priors, are not `improper' but may reflect complete or partial ignorance
about a parameter after conditioning on the data.  Returning to the above Poisson-process
example, at time $t_1$, we have clearly learned something about $\lambda$ in that
our belief in large values of the Poisson intensity $\lambda$ has decreased while our
relative degree of belief in small values of $\lambda$ has remained approximately unchanged.
That the posterior is improper do not imply that our prior was wrong,
but only that more data perhaps needs to be collected if possible.
Proceeding by using the improper posterior at time $t_1$
as prior in subsequent inference, say based on the number of occurrences  observed
in a sufficiently long subsequent interval $(t_1,t_2]$, we indeed eventually reach
the same proper final posterior as the one reached by combining the initial scale
prior and the likelihood for the data on $(0,t_2]$.
We hope that the reader can appreciate that this argument 
indicates also the potential philosophical importance of unbounded laws more generally.

The most influential initial work on Bayesian inference is
given by the book of \citet{JEFFREYS}.
Parts of his arguments were mainly intuitive, 
and there is
a lack of mathematical rigor as also observed by \citet{RobertChopinRousseau09jeffreys}.
The needed mathematical theory for a rigorous reformulation of the original 
arguments of \citet{JEFFREYS} is presented next.

\vekk{

The most familiar example of an improper law $\pr_{\Theta}$ is
the uniform law on the real line $\Omega_\Theta = \RealN$.
Following \citet{RENYI} and \citet{TaraldsenLindqvist16renyi} 
the uniform law is identified with
the countable collection of uniform laws
$\Pr_\Theta (\cdot | B_n)$ on each interval $B_n = [-n, n]$.
This gives then also the interpretation of $\pr_{\Theta}$:
Given that $\Theta \in B_n$ the law is the uniform probability 
distribution on $B_n$.
The family $\cB = \{B_n \st n \in \NatN\}$ is a {\em bunch} and
the family $\{\pr_\Theta (\cdot \st B) \st B \in \cB\}$ defines a Rényi space.
The concept of a Rényi space and other elements from measure theory
are defined in Appendix~\ref{sDef}.

The aim of this paper is to present a theory of statistics that
allows improper laws as priors.
This extends the results of 
\citet{TaraldsenLindqvist10ImproperPriors,TaraldsenLindqvist16renyi},
and provides stronger links to the results on improper
laws presented by \citet{HARTIGAN} and \citet{BiocheDruilhet16convergence}.
The main result is Theorem~\ref{theo1} which
proves existence and uniqueness of conditional expectation
on Rényi spaces.
The existence and uniqueness proof relies on the Radon-Nikodym theorem
and a generalization of the Rényi structure theorem.

The theory of conditional expectation has been most important
for the development of measure theory, probability, and statistics
based on Kolmogorov's concept of a probability space.
The generalization of this to the setting of Rényi spaces can hence
be expected to be important for future developments. 
}


\section{Existence and uniqueness conditional expectation}
\label{sTheory}

\citet{TaraldsenLindqvist10ImproperPriors,TaraldsenLindqvist16renyi}
define the posterior law $\pr (A \st X=x)$ for the case where the
data $X$ is $\sigma$-finite.
The aim now is to prove existence and uniqueness
of a posterior law without assuming that $X$ is $\sigma$-finite.
It will be convenient to do this as an extension of conventional
measure theory, and then get the result for the posterior law as a special case.
The reader is advised to consult Appendix~\ref{sDef} for the definition
of a Rényi space and other elements from measure theory if needed.

Let $(\setX, \cF)$ be a measurable space and let
$(\setT, \cG, \nu)$
be a measure space \citep{RUDIN}.
Let a measurable function $\setT \ni t \mapsto \mu^t (A) \in [0,\infty]$
be given for each $A \in \cF$.
We define $\mu$ to be a {\it strong random measure} on $\cF$ with respect to the law $\nu$ if
the following holds for all disjoint measurable $A_1, A_2, \ldots$
\begin{enumerate}
\item $\mu^t (\emptyset) = 0$ for almost all $t$.
\item $\mu^t (A_1 + A_2 + \cdots) = \mu^t (A_1) + \mu^t (A_2) + \cdots $ for almost all $t$.
\end{enumerate}
The notation $A + B$ denotes the union of two disjoint and measurable sets.
Equality for almost all $t$ means that equality holds for all $t$
in a set $E$ where $\nu (E^c) = 0$.
All functions of $t$ here and in the following are assumed only to 
be defined almost everywhere,
and the set $E$ corresponding to $E \ni t \mapsto \mu^t (A_k)$ depends on $A_k$.

If there exists $B_1, B_2, \ldots$ with $\setX = \cup_n B_n$ and $0 < \mu^t (B_n) < \infty$,
then $\mu$ is said to be $\sigma$-finite.
If, additionally, $\mu^t$ is a measure on
$\cF$ for all $t$,
then $\mu$ is a random measure. 
In this paper we define the space $\setX$ to be regular if
every $\sigma$-finite strong random measure can be represented by a random measure.
The Borel $\sigma$-algebra of a complete separable metric space $\setX$
gives a regular space.

The concept of a strong random measure is here introduced
similarly to how \citet[p.1-2]{Skorohod84} introduces the concept of
a strong random operator.
He also defines the notion of a weak random operator by duality,
but this is equivalent with strong when the image space is the complex numbers.
It should be noted that the naming convention here is
counter intuitive in the sense that a strong random operator is a weaker
concept than a random operator,
but there are good reasons for adopting the conventions of Skorohod.

Let $\mu$ be a measure on $\setX$ and let
$T: \setX \into \setT$ be measurable.
We define a strong random measure $\mu^t (A)$ to be 
the conditional law of $\mu$ given $T = t$ if
\be{1}
\mu (A [T \in C] \st B) = \int_C \mu^t (A \st B) \, \mu_T (dt \st B)
\ee
for all measurable $A,C$ and
all $B$ with $0 < \mu(B) < \infty$,
where
$\mu^t (A \st B) = \mu^t (A B) / \mu^t (B)$,
$\mu_T (D \st B) = \mu (T \in D \st B)$,
and $\mu (F \st B) = \mu (F B)$ (no normalization here!).
The notation
$(T \in D) = \{x \st T(x) \in D\} = T^{-1} (D)$ is similar to 
the notation $\{T (x) \in D\}$ used by \citet[p.1]{DOOB}.
For a conditional law we use the notation $\mu^t (A) = \mu (A \st T=t)$.
The double use of the symbol $\mu$ in the above is justified by:
\begin{theo}
\label{theo1}  
A unique $\sigma$-finite conditional law $\mu (A \st T=t)$ exists 
if $\mu$ is a $\sigma$-finite measure on
$\setX$ and $T: \setX \into \setT$ is measurable.
\end{theo}
\begin{proof}

If $t \mapsto c(t) > 0$ is measurable,
then  $\mu^t (A B) / \mu^t (B) = c (t) \mu^t (A B) / [c(t) \mu^t (B)]$
shows that uniqueness up to multiplication by a positive function is the best possible uniqueness.
The measure $C \mapsto \mu (A B [T \in C])$ is dominated by the
$\sigma$-finite measure $C \mapsto \mu ([T \in C] B)$,
so a unique normalized strongly random $\mu^t (A \st B)$ follows from the Radon-Nikodym theorem.
It remains to prove that $\mu^t (A \st B) = \mu^t (A B) / \mu^t (B)$ for
a $\sigma$-finite strong random measure $\mu^t$.

Let $B_1 \subset B_2 \subset \cdots$ with $0 < \mu(B_n) < \infty$ and $\cup_n B_n = \setX$,
and define $\mu^t (B_n) = \mu^t (B_n \st B_1 \cup B_n) / \mu^t (B_1 \st B_1 \cup B_n)$.
It follows that $\mu^t (B_n) = 1/\mu^t (B_1 \st B_n)$ and
$\mu^t (B_1) = 1$.
For $A \subset B_n$ put $\mu^t (A) = \mu^t (A \st B_n) \mu^t (B_n)$
and define $\mu^t (A) = \mu^t (A \cap B_1) + \sum_{n\ge 2} \mu^t (A \cap B_n \cap B_{n-1}^c)$
for a general $A \in \cF$.

It must be proved that the above construction is well defined.
Let $A \subset B_n \subset B_m$.
It must be proved that
(*) $\mu^t (A \st B_n) \mu^t (B_n) = \mu^t (A \st B_m) \mu^t (B_m)$.
Observe first that
$\mu (B_n B_m [T \in D]) = \int_D \mu^t (B_n \st B_m) \mu_T (dt \st B_m)$
$= \int_D \mu_T (dt \st B_n)$ gives
$\mu_T (dt \st B_n) = \mu^t (B_n \st B_m) \mu_T (dt \st B_m)$.
The (*) claim follows from
\be{2}
\mu (A (T \in D)) = \int_D \mu^t (A \st B_n) \mu^t (B_n \st B_m) \mu_T (dt \st B_m)
= \int_D \mu^t (A \st B_m) \mu_T (dt \st B_m)
\ee
since $\mu^t (B_n \st B_m) = \mu^t (B_n) / \mu^t (B_m) = \mu^t (B_1 \st B_m) / \mu^t (B_1 \st B_n)$
follows from the case $A = B_1$ in equation~(\ref{eq2}).
This defines a unique $\mu^t (A)$.
The remaining claims follow by
verification and is left to the reader.
\end{proof}

All of the previous can be repeated with a replacement of
the measurable set $A$ with a positive measurable function
$A: \setX \into [0,\infty]$ and conditional
expectation of complex valued functions can be defined
by decomposition in positive and negative parts and then in real 
and complex parts.
Consideration of the dual space gives conditional expectation of
a separable Banach space valued $A: \setX \into \setB$.
The conditional expectation $\E (A \st T=t)$
is in particular well defined when $A$ is a separable Hilbert space valued.
Separability is assumed to ensure almost everywhere definition
on $\setT$.

An alternative approach, as noted by \citet[p.54, eq.10]{KOLMOGOROV},
is to define the conditional expectation
by integration with respect to the conditional probability.
For $\setB$ equal to the set of real numbers $\RealN$
or the set of complex numbers $\CompN$ the alternative approach
gives the same strong random linear functional,
but for more general $\setB$ there are many alternative routes with
different results.

Conditional expectation $\mu (A \st \cT)$ with respect
to a $\sigma$-field $\cT \subset \cF$
is defined by $T (x) = x$ and
$(\setT, \cG) = (\setX, \cT)$.
It can be noted that we define
$\mu^t (A)$ directly instead of more indirectly
by first defining $\mu (A \st \cT)$ as is more common.
This has the advantage
of allowing a completely general measurable space $\setT$,
whereas the common approach requires separability
properties of $\setX$ according to \citet[p.616, Prop.B.24]{SCHERVISH}.

It can finally be observed that the proof of Theorem~\ref{theo1}
contains a proof of a structure theorem for
strong random Rényi spaces defined by
a consistent family of strong random conditional
probabilities $\mu^t (A \st B)$ for $B \in \cB$ for a fixed bunch $\cB \subset \cF$:
The family is generated by a strong random measure $\mu^t (A)$
such that $\mu^t (A \st B) = \mu^t (A B) / \mu^t (B)$.
The consistency requirement is that
$B_1 \subset B_2$ implies $\mu^t (B_1 \st B_2) > 0$ and
\be{RanReny}
\mu^t (A \st B_1) = \frac{\mu^t (A B_1 \st B_2)}{\mu^t (B_1 \st B_2)}
\ee

\section{Examples}

\subsection{Mathematical statistics}

A statistical model is given by the structure
%
\be{StatModFirst}
\scalebox{1.2}{
\begin{tikzcd}[row sep=normal, ampersand replacement=\&]
\&\Omega_\Theta \arrow[r, "\psi"] \& \Omega_\Gamma  \\
(\Omega, {\cal E}, \pr) \arrow[ur, "\Theta"] 
\arrow[drr, "Y" near end] \arrow[dr, "X"'] 
\arrow[urr, "\Gamma"' near end] \& \& \\
\&\Omega_X \arrow[r, "\phi"'] \& \Omega_Y
\end{tikzcd}
}
\ee
In conventional theory \citep{SCHERVISH} the space $\Omega$
is a probability space.
In the more general setting of a Rényi space $\Omega$ considered here
the underlying law $\pr$ is a conditional probability law with
a corresponding bunch $\cB \subset \cE$.
The law of the data $X$ given the parameter $\Theta = \theta$
is defined by $\pr_X^\theta (A) = \pr (X \in A \st \Theta = \theta)$.
The law of the model parameter $\Theta$ given the data $X=x$
is defined by $\pr_\Theta^x (A) = \pr (\Theta \in A \st X = x)$.
The posterior law $\pr_\Gamma^x$ of a parameter $\gamma = \psi (\theta)$
and the law $\pr_Y^\theta$ of a statistic are determined by this
and equation~(\ref{eqStatModFirst}).

In previous work by
\citet{TaraldsenLindqvist10ImproperPriors,TaraldsenLindqvist16renyi,LindqvistTaraldsen17improper}
it was required that the data $X$ is $\sigma$-finite,
but Theorem~\ref{theo1} shows that this requirement is not needed.
There is, however, a prize:
The posterior must in general be interpreted as an improper law.
The initial examples demonstrate, however, 
that even if the data $X$ is not $\sigma$-finite,
it may happen that the posterior is a proper distribution for some
values of $x$.
Assuming that $X$ is $\sigma$-finite ensures
that the posterior is always proper.

A similar comment holds for the law
$\pr_X^\theta$ of the data $X$.
The factorization $\pr_{X,\Theta} (dx,d\theta) = \pr_X(dx \st \Theta = \theta) \pr_\Theta (d\theta)$
holds uniquely if and only if $\Theta$ is $\sigma$-finite.
In this case, a $\sigma$-finite $\Theta$ is required if
the most common Bayesian recipe is to be used:
\begin{enumerate}
\item Specify a statistical model law $\pr_X^\theta$
\item Specify a prior $\pr_\Theta$
\item Compute the posterior $\pr_\Theta^x$  
\end{enumerate}
If $\Theta$ is not $\sigma$-finite,
the first two steps must be replaced by a direct
specification of a joint $\sigma$-finite law
$\pr_{X,\Theta}$,
and Theorem~\ref{theo1} ensures then that the
posterior $\pr_\Theta^x$ is uniquely defined.

\subsection{Densities}

Assume that $(X,\Theta) \sim f(x,\theta) \mu (dx) \nu (d\theta)$.
It follows that
$(X \st \Theta = \theta) \sim f(x, \theta) \mu (dx)$,
and that $(\Theta \st X = x) \sim f(x, \theta) \nu (d\theta)$.
This can be verified directly by the defining equation~(\ref{eq1}).
It follows in particular that this is consistent with
the definition of an improper posterior
used by \citet[p.1716]{BiocheDruilhet16convergence}.
The previous can also be reformulated as
$f (x \st \theta) = f(x, \theta) = f(\theta \st x)$:
There is no need for a normalization constant since two
proportional densities are equivalent! 

Let $c(\theta) > 0$ be an otherwise arbitrary measurable function.
It follows then that
$(X \st \Theta = \theta) \sim c (\theta) f(x, \theta) \mu (dx)$
when interpreted as a strong random conditional law.
A formal prior density $\pi$
gives the joint density $c (\theta) f(x, \theta) \pi (\theta)$,
and this shows that the interpretation of
$\pi$ as prior information is dubious in this case
as pointed out by \citet{LavineHodges12icar} and \citet{LindqvistTaraldsen17improper}
using a model with intrinsic conditional auto-regressions.
It is the resulting joint distribution and the
conditional laws that can be interpreted.
The usual decomposition in a prior and a model
can only be interpreted uniquely if the prior for the 
model parameter $\Theta$ is $\sigma$-finite.
Conversely, \citet{LindqvistTaraldsen17improper} obtain
a posterior only in the case where the data $X$ is $\sigma$-finite,
but Theorem~\ref{theo1} ensures that a posterior is defined also without requiring
a $\sigma$-finite $X$.

A concrete simple example is given by letting $\pr_{X,\Theta}$ correspond to
Lebesgue measure in the plane.
The law of $X$ given $\Theta = \theta$ and the
posterior law of $\Theta$ given $X=x$ correspond then both
to Lebesgue measure on the line.
The factorization $f(x,\theta) = 1 = c(\theta) \pi (\theta)$
with $\pi(\theta) = 1/c(\theta)$ is completely arbitrary.
This can be interpreted according to \citet[p.26]{HARTIGAN}
as saying that the marginal distribution $\pi$ is not determined
by the joint distribution.
This interpretation is discussed in more detail by
\citet{TaraldsenLindqvist10ImproperPriors},
but it differs from the interpretation here.
The marginal law of $\Theta$ is unique,
but given by the non-$\sigma$-finite measure
$\pr_\Theta (d\theta) = \infty \cdot d\theta$.
It follows in particular that the decomposition
$\pr_{X,\Theta} (dx,d\theta) = \pr_X^\theta (dx) \pr_\Theta (d\theta)$
fails in this case.

Let $\mu (dx,dy)$ be Lebesgue measure in the plane,
and consider the indicator function of the upper half plane:
$T (x,y) = [y > 0]$.
It follows that $\mu_T (dt) = [\infty \delta_0 (t) + \infty \delta_1 (t)] dt$
so $T$ is not $\sigma$-finite.
The conditional law
$\mu^t (dx,dy) = [(t=1)(y>0) + (t=0)(y \le 0)] dx dy$ is, however,
well defined.
The conditional law $\mu^1$ is Lebesgue measure restricted to the
upper half-plane and $\mu^0$ is Lebesgue measure restricted to
the lower half plane.
This demonstrates directly that
the conditional law is also defined when $T$ is not $\sigma$-finite.

Consider more generally
a function $T: \setX \into \NatN$.
This gives
$\mu^t (A) = \mu (A (T=t))$ and then also
the elementary definition of the law 
$\mu (A \st B) = \mu (A B)$ for any $B$ with
$\mu (B) > 0$.
This is consistent with the familiar
$\mu (A \st B) = \mu(A B)/\mu(B)$
for the case where $\mu (B) < \infty$.
A law is arbitrary up to multiplication by a positive constant:
It is an equivalence class of $\sigma$-finite measures.
Theorem~\ref{theo1} gives the existence of conditional expectations
in full generality - including this elementary case.

\subsection{The marginalization paradox}

\citet[p.370]{StoneDawid72} consider inference for the ratio $\theta$ of two exponential means.
They assume that $X$ and $Y$ are independent exponentially distributed with
hazard rates $\theta \phi$ and $\phi$ respectively.
It is then clear that $Z = Y/X$ will have a distribution that only depends on $\theta$.
In fact, $Z = \theta F$, where $F$ has a Fisher distribution with $2$ and $2$
degrees of freedom since a standard exponential variable is distributed
like a $\chi^2_2 /2$ variable.
It follows then that the density is
\be{mm1}
f (z \st \theta) = \theta^{-1} (1 + z/\theta)^{-2} = \theta (\theta + z)^{-2}
\ee
The posterior density corresponding to a prior density $\pi (\theta)$ is then
\be{mm2}
\pi (\theta \st z) \propto \frac{\theta \pi (\theta)}{(\theta + z)^{2}}
\ee
A different argument goes as follows.
The joint density with a joint prior $\pi (\theta) d\theta d\phi$ gives
%
$
\pi (\theta, \phi \st x, y) \propto
\pi (\theta) \theta \phi^2 \exp(-\phi (\theta x + y))
$
%
The marginal posterior of $\theta$ follows by integration over $\phi$
which gives
\be{mm4}
\pi (\theta \st x, y)
\propto \frac{\theta \pi (\theta)}{(\theta x + y)^{3}}
\propto \frac{\theta \pi (\theta)}{(\theta + z)^{3}}
\ee
Equation~(\ref{eqmm4}) gives a posterior given the data $(x,y)$
that differs from the posterior found in equation~(\ref{eqmm2}).
This constitutes the argument and paradox presented originally by
\citet[p.370]{StoneDawid72}.

A range of similar paradoxes were presented later by
\citet{DawidStoneZidek73} with discussion of links to fiducial inference.
They claim that the \citet{FRASER} theory of structural inference
is intrinsically paradoxical under marginalization.
Furthermore, Lindley,
in his discussion of the paper
\citep[p.218]{DawidStoneZidek73} writes:
\begin{quote}
{\it The paradoxes displayed
here are too serious to be ignored and impropriety must go.
Let me personally retract
the ideas contained in my own book.
}
\end{quote}  
This is of particular relevance here since in 1964,
in his book,
\citet[p.xi]{Lindley80BayesBook} wrote:
\begin{quote}
{\it
  The axiomatic structure used here  is  not the usual one
  associated with the name of Kolmogorov.
  Instead one based on the
ideas  of  Rényi has been used.
}
\end{quote}  
We argue here and in the following that Lindley's initial intuition
was correct:
The theory of Rényi gives a mathematical foundation
for statistics that allows unbounded measures.

\subsection{Resolving the paradox}

We disagree with the criticism of Fraser's structural inference,
but more importantly we will next explain that there is
no paradox related to the above problem when it is treated
within the theory of Rényi.
This has already been indicated by \citet{TaraldsenLindqvist10ImproperPriors},
and is discussed in more detail by \citet{LindqvistTaraldsen17improper}.
\citet{LindqvistTaraldsen17improper} rely on a theory where
it is only allowed to condition on $\sigma$-finite statistics.
We extend this argument now with reference to Theorem~\ref{theo1}
which allows conditioning on any statistic.

The initial assumptions are interpreted to imply a
joint distribution given by the density:
\be{mm5}
f (x,z,\theta,\phi) = \pi(\theta) \theta \phi^2 x e^{-\phi x (\theta + z)} 
\ee
\citet{TaraldsenLindqvist10ImproperPriors} explain that any marginal is
determined by a joint density and
integration over $\phi$ gives then
\be{mm6}
f (x,z,\theta) = \pi(\theta) \theta x^{-2} (\theta + z)^{-3}
\ee
which implies
\be{mm7}
f (\theta \st x,z) =
\pi(\theta) \theta x^{-2} (\theta + z)^{-3} =
\pi(\theta) \theta (\theta + z)^{-3}
\ee
The second equality holds since it is equality in the sense
given by an equivalence class as in Theorem~\ref{theo1}.
The right hand side can be multiplied by an arbitrary positive function $c(x,z)$
without changing the equality sign.
Equation~(\ref{eqmm7}) is equivalent with equation~(\ref{eqmm4}) since
there is a one-one correspondence between
$(x,y)$ and $(x,z)$.

An alternative is to integrate over $x$ to obtain
\be{mm8}
f (z,\theta, \phi) = \pi(\theta) \theta (\theta + z)^{-2}
\ee
which implies
\be{mm9}
f (\theta \st z, \phi) = \pi(\theta) \theta (\theta + z)^{-2}
\ee
This is similar to equation~(\ref{eqmm2}),
but it is different since equation~(\ref{eqmm2})
only condition on $z$.
Equation~(\ref{eqmm9})
is not in conflict with
equation~(\ref{eqmm7}) for the same reason.

Starting with either equation~(\ref{eqmm6}) or
equation~(\ref{eqmm8}) gives
\be{mm10}
f (z,\theta) = \infty \cdot \pi(\theta)
\ee
which shows that neither
$Z \st \Theta=\theta$
nor
$\Theta \st Z=z$
can be represented by a $\sigma$-finite measure.
This implies that the argument in equations~(\ref{eqmm1}-\ref{eqmm2})
is wrong given equation~(\ref{eqmm5}).
Equation~(\ref{eqmm4}), or equivalently
equation~(\ref{eqmm7}), gives the correct posterior distribution for $\Theta$.

If, instead, the prior $\pi(\theta) \phi^{-1} d\theta d\phi$ is used,
then the result will be
\be{mm11}
f (\theta \st z, \phi) = \phi^{-1} \pi(\theta) \theta (\theta + z)^{-2}
= \pi(\theta) \theta (\theta + z)^{-2}
\ee
and
\be{mm12}
f (\theta \st x,z) = \pi(\theta) \theta x (x (\theta + z))^{-2}
= \pi(\theta) \theta (\theta + z)^{-2}
\ee
The conditionals coincide,
but it is still true that neither equals the law of $\Theta \st Z=z$
since the law of $(\Theta, Z)$ still fails to be $\sigma$-finite.

If, however, equation~(\ref{eqmm1}) together with a prior
$\pi (\theta)$ is taken as the initial $\sigma$-finite law for $(\Theta,Z)$,
then equation~(\ref{eqmm2}) is the correct posterior.
The paradox is then removed since the conflicting conclusions
are consequences of different initial assumptions.

It can finally be noted that
even if a conditional law $\pr^{x,z}$ does not depend on $x$
it can not be concluded that it equals $\pr^z$.
This is demonstrated by equation~(\ref{eqmm7}) and
equation~(\ref{eqmm9}).
This rule holds for probability distributions,
and also more generally if $Z$ and $(X,Z)$ are $\sigma$-finite.
\citet{StoneDawid72} calculated formally as explained above
as if the rule where generally valid.
This error resulted in two conflicting results.

\vekk{
\subsection{Toy example}
Improper target density 
\begin{equation}
f(x) = \frac1{\sqrt{2\pi}}e^{-x^2/2}+\frac{f_\infty}{1+e^{-x}}
\end{equation}
This has an improper cdf
\begin{equation}
F(x) = \phi(x) + f_\infty \ln(e^x+1).
\end{equation}
Letting $a_i$, $i=0,2,\dots,m$ denote the boundaries between the intervals 
where $a_0=-\infty$, the theoretically optimal kernel weights for each interval are then
\begin{equation}
w_i = F(a_i)-F(a_{i-1})
\end{equation}

\subsection{Butterflies}
\citep{Tufto2012butterflies}
\subsection{Random effects}

\subsection{Removal sampling}
\citep{Druilhet2016}
}

\section{Final remarks}

It follows from the previous quotations of Lindley that he
initially supported the use of conditional probability spaces
as introduced by Rényi.
We have argued that this initial suggestion is indeed a natural approach
to Bayesian statistics including commonly used objective priors.

As explained, the marginalization paradoxes seem to have been the main reason
for Lindley's change in opinion on this.
Tony O'Hagan \href{https://www.youtube.com/watch?v=cgclGi8yEu4}{interviewed}
Lindley for the Royal Statistical Society's Bayes 250
Conference held in June 2013.
Lindley explains very nicely that all probabilities are conditional probabilities,
but also recalls his reaction to the marginalization paradoxes:

\begin{quote}
{\it Good heavens, the world is collapsing about me.} 
\end{quote}

Lindley continuous to argue that Bayesian statistics without
improper priors is a sound theory, 
and that the focus should be on
how to quantify the prior uncertainty of the
unknown parameters.
The parameters should be viewed as real physical quantities 
regardless of which experiment is later used for decreasing their uncertainty.
This clearly disqualifies the choice of data dependent priors,
and even the choice of priors depending on the particular statistical model used.
We wholeheartedly agree with Lindley on this,
but we claim that this can be done also within the more general
theory introduced by Rényi and continued here.

An unbounded law can according to Rényi be interpreted by
the corresponding family of conditional probabilities given
by conditioning on the events in the bunch.
These elementary conditional probabilities are probabilities in the sense of Kolmogorov,
and the interpretation depends on the application.
They can, as Lindley advocates convincingly, be interpreted as personal probabilities
corresponding to a range of real life events.
They can also, as needed in for instance quantum physics, be interpreted as as objectively true probabilities
representing a law for how a system behaves when observed repeatedly under idealized conditions.

Assume now that we accept a theory where the prior uncertainty is given by a possibly unbounded law.
It is then natural to accept that a resulting posterior uncertainty can also be
represented by a possibly unbounded law.
Both the prior and the posterior represent uncertainty of the same kind.
Hopefully, many can agree on this on an intuitive level.
The main result presented here is Theorem~\ref{theo1} which provides a mathematical theory
in which this can be done consistently without paradoxical results.

\appendix

\section{Appendix on measure theory}
\label{sDef}

\subsection{Measurable space and measure}

A measurable space is a set $\setX$ equipped with
a $\sigma$-field $\cF$ of subsets of $\setX$.
A $\sigma$-field $\cF$ is a collection of
subsets of a fixed set that contains the empty set $\emptyset$
and is closed under complements and countable unions.
A set $A \subset \setX$ is measurable if $A \in \cF$.
A measure $\mu$ is a function
$\mu: \cF \into [0,\infty]$ with $\mu (\emptyset) = 0$
that is countably additive:
$\mu (A_1 + A_2 + \cdots) = \mu (A_1) + \mu (A_2) + \cdots$.
A probability measure is a measure $\mu$ on
a measurable space $\setX$ with $\mu (\setX) = 1$.
A measure space is a measurable space $\setX$ equiped with a
measure \citep[p.16]{RUDIN}. 
A probability space is a measurable space $\setX$ equipped with a
probability measure.


A sigma-field $\cF_0 \subset \cF$ of a measure
space $(\setX, \cF, \mu)$ is sigma-finite
if there exist measurable sets $F_1, F_2, \ldots \in \cF_0$ with
$\mu (F_i) < \infty$ and $\setX = F_1 \cup F_2 \cup \cdots$ 
\citep[p.5010]{TaraldsenLindqvist16renyi}.
A measure space $(\setX, \cF, \mu)$ is sigma-finite if
$\cF$ is sigma-finite,
and $\mu$ is then also said to be sigma-finite
\citep[p.121]{RUDIN}. 

\subsection{Conditional probability space}

A bunch $\cB$ in a measurable space is a family of
measurable sets closed under finite unions that does not contain
the empty set, but contains a countable family $F_1, F_2, \ldots$ of sets
whos union is the whole set.
A bunch $\cB$ is ordered if $B_1,B_2 \in \cB$ implies
$B_1 \subset B_2$ or $B_2 \subset B_1$.

A Rényi space \citep[p.5013]{TaraldsenLindqvist16renyi} 
is a measurable space $\setX$ equipped with a family 
$\{\nu (\cdot \st B) \st {B \in \cB}\}$ 
of probability measures indexed by a bunch $\cB$ which fulfill   
$B_1, B_2 \in \cB$ and $B_1 \subset B_2$ 
$\imply \nu (B_1 \st B_2) > 0$, 
and the identity
\be{RenyiSpace}
\nu (A \st B_1) = \frac{\nu (A \cap B_1 \st B_2)}{\nu (B_1 \st B_2)}
\ee
A sigma-finite measure $\mu$ on a measurable space $\setX$
generates a probability law
$\nu = [\mu] = \{c \mu \st c \in \RealN_+ \}$ with
corresponding conditional probabilities
$\nu (A \st B) = \mu(A \cap B)/\mu (B)$ for
$B \in cB = \{B \st 0 < \mu (B) < \infty\}$.
The Rényi space generated by the probability law $\nu$
is given by the family
$\{\nu (\cdot \st B) \st {B \in \cB}\}$.
A conditional probability space is a set $\setX$ equipped with a probability law.

\subsection{Statistical model}

A statistical model is a triple $(\Omega, X, \Theta)$
where the space $\Omega$ is a conditional probability space,
the data $X$ is a measurable function
$X: \Omega \into \Omega_X$, 
and the model parameter
$\Theta$ is a measurable function 
$\Theta: \Omega \into \Omega_\Theta$.
These definitions,
and the ones that follow,
are as given by \citet{SCHERVISH} except for choice of symbols
and the generalization given by assuming
that $\Omega$ is a conditional probability space.
There is one probability law $\pr$ defined on the sigma-algebra
$\cE$ of events in $\Omega$ - and all other concepts
are defined from the basic space $\Omega$
\citep[p.5011]{TaraldsenLindqvist16renyi}.

A statistic 
$Y = \phi(X) = \phi \circ X$ 
is a measurable function of the data 
and a parameter 
$\Gamma = \psi(\Theta) = \psi \circ \Theta$
is a measurable function
of the model parameter as
illustrated in equation~(\ref{eqStatMod})
\be{StatMod}
\scalebox{1.2}{
\begin{tikzcd}[row sep=normal, ampersand replacement=\&]
\&\Omega_\Theta \arrow[r, "\psi"] \& \Omega_\Gamma  \\
(\Omega, {\cal E}, \pr) \arrow[ur, "\Theta"] 
\arrow[drr, "Y" near end] \arrow[dr, "X"'] 
\arrow[urr, "\Gamma"' near end] \& \& \\
\&\Omega_X \arrow[r, "\phi"'] \& \Omega_Y
\end{tikzcd}
}
\ee
A random quantity is
a measurable function $Z: \Omega \into \Omega_Z$,
and its law is defined by 
$\pr_Z (A) = \pr (Z \in A)$ where 
$(Z \in A) = \{\omega \st Z(\omega) \in A\}$.
We abuse notation here and interpret $\pr$
as one fixed representative of the equivalence class that
defines $\pr$ as a conditional measure. 
A random quantity is sigma-finite if its law is
sigma-finite. 
If the model parameter $\Theta$ is sigma-finte, 
then the conditional probabilities
$\pr^\theta_X (A) = \pr (X \in A \st \Theta = \theta)$
define a family of probability measures on
the sample space $\Omega_X$ indexed by
the model parameter $\theta$ in the 
model parameter space $\Omega_\Theta$. 
Likewise, if the data $X$ is sigma-finite,
then the posterior 
$\pr^x_\Theta (B) = \pr (\Theta \in B \st X = x)$
is a probability measure
on the model parameter space $\Omega_\Theta$.
The mappings
$\theta \mapsto \pr_X^\theta (A)$ and
$x \mapsto \pr_\Theta^x (B)$ are measurable for all events $A$,
but existence of families of probability measures
as claimed above requires
regularity assumptions:
It is sufficient to assume that
the sample space $\Omega_X$ and the model parameter space $\Omega_\Theta$ 
are Borel spaces \citep[p.619]{SCHERVISH}\citep[p.5011]{TaraldsenLindqvist16renyi}.

\addcontentsline{toc}{section}{References} 

\bibliographystyle{chicago}
\bibliography{gt,jarle,bib,gtaralds}

\begin{thebibliography}{}

\bibitem[\protect\citeauthoryear{Bioche and Druilhet}{Bioche and
  Druilhet}{2016}]{BiocheDruilhet16convergence}
Bioche, C. and P.~Druilhet (2016).
\newblock {Approximation of improper priors}.
\newblock {\em Bernoulli\/}~{\em 3\/}(22), 1709--1728.

\bibitem[\protect\citeauthoryear{Dawid, Stone, and Zidek}{Dawid
  et~al.}{1973}]{DawidStoneZidek73}
Dawid, A.~P., M.~Stone, and J.~V. Zidek (1973).
\newblock {Marginalization Paradoxes in Bayesian and Structural Inference}.
\newblock {\em Journal of the Royal Statistical Society Series B-Statistical
  Methodology\/}~{\em 35\/}(2), 189--233.

\bibitem[\protect\citeauthoryear{Doob}{Doob}{1953}]{DOOB}
Doob, J.~L. (1953).
\newblock {\em {Stochastic Processes}}.
\newblock Wiley Classics Library Edition (1990). Wiley.

\bibitem[\protect\citeauthoryear{Druilhet, Bioche, and Druilhet}{Druilhet
  et~al.}{2016}]{Druilhet2016}
Druilhet, P., C.~Bioche, and P.~Druilhet (2016).
\newblock {A cautionary note on Bayesian estimation of population size by
  removal sampling with diffuse priors}.
\newblock {\em Biometrical Journal\/}~{\em 00\/}(0000).

\bibitem[\protect\citeauthoryear{Fraser}{Fraser}{1968}]{FRASER}
Fraser, D. A.~S. (1968).
\newblock {\em {The structure of inference}}.
\newblock John Wiley.

\bibitem[\protect\citeauthoryear{Haldane}{Haldane}{1932}]{Haldane32prior}
Haldane, J. B.~S. (1932).
\newblock A note on inverse probability.
\newblock {\em Mathematical Proceedings of the Cambridge Philosophical
  Society\/}~{\em 28}, 55--61.

\bibitem[\protect\citeauthoryear{Hartigan}{Hartigan}{1983}]{HARTIGAN}
Hartigan, J. (1983).
\newblock {\em {Bayes theory}}.
\newblock New York: Springer.

\bibitem[\protect\citeauthoryear{Jeffreys}{Jeffreys}{1939}]{JEFFREYS}
Jeffreys, H. (1939).
\newblock {\em {Theory of probability (1966 ed)}\/} (Third ed.).
\newblock New York: Oxford.

\bibitem[\protect\citeauthoryear{Kolmogorov}{Kolmogorov}{1933}]{KOLMOGOROV}
Kolmogorov, A. (1933).
\newblock {\em {Foundations of the theory of probability}\/} (Second ed.).
\newblock Chelsea edition (1956).

\bibitem[\protect\citeauthoryear{Lavine and Hodges}{Lavine and
  Hodges}{2012}]{LavineHodges12icar}
Lavine, M.~L. and J.~S. Hodges (2012).
\newblock {On Rigorous Specification of ICAR Models}.
\newblock {\em The American Statistician\/}~{\em 66\/}(1), 42--49.

\bibitem[\protect\citeauthoryear{Lindley}{Lindley}{1980}]{Lindley80BayesBook}
Lindley, D.~V. (1980, March).
\newblock {\em {Introduction to Probability and Statistics from a Bayesian
  Viewpoint, Part 1, Probability (Pt. 1)}}.
\newblock Cambridge University Press.

\bibitem[\protect\citeauthoryear{Lindqvist and Taraldsen}{Lindqvist and
  Taraldsen}{2017}]{LindqvistTaraldsen17improper}
Lindqvist, B. and G.~Taraldsen ((in press) 2017).
\newblock On the proper treatment of improper distributions.
\newblock {\em J. Statist. Plann. Inference\/}.

\bibitem[\protect\citeauthoryear{Renyi}{Renyi}{1970}]{RENYI}
Renyi, A. (1970).
\newblock {\em {Foundations of Probability}}.
\newblock Holden-Day.

\bibitem[\protect\citeauthoryear{Robert, Chopin, and Rousseau}{Robert
  et~al.}{2009}]{RobertChopinRousseau09jeffreys}
Robert, C.~P., N.~Chopin, and J.~Rousseau (2009).
\newblock {Harold Jeffreys's Theory of Probability Revisited}.
\newblock {\em Statistical Science\/}~{\em 24\/}(2), 141--172.

\bibitem[\protect\citeauthoryear{Rudin}{Rudin}{1987}]{RUDIN}
Rudin, W. (1987).
\newblock {\em {Real and Complex Analysis}}.
\newblock McGraw-Hill.

\bibitem[\protect\citeauthoryear{Schervish}{Schervish}{1995}]{SCHERVISH}
Schervish, M.~J. (1995).
\newblock {\em Theory of {S}tatistics}.
\newblock Springer.

\bibitem[\protect\citeauthoryear{Skorohod}{Skorohod}{1984}]{Skorohod84}
Skorohod, A.~V. (1984, November).
\newblock {\em {Random Linear Operators}\/} (1 ed.).
\newblock Springer.

\bibitem[\protect\citeauthoryear{Stone and Dawid}{Stone and
  Dawid}{1972}]{StoneDawid72}
Stone, M. and A.~P. Dawid (1972).
\newblock {Un-Bayesian Implications of Improper Bayes Inference in Routine
  Statistical Problems}.
\newblock {\em Biometrika\/}~{\em 59\/}(2), 369--375.

\bibitem[\protect\citeauthoryear{Taraldsen and Lindqvist}{Taraldsen and
  Lindqvist}{2010}]{TaraldsenLindqvist10ImproperPriors}
Taraldsen, G. and B.~H. Lindqvist (2010).
\newblock {Improper Priors Are Not Improper}.
\newblock {\em The American Statistician\/}~{\em 64\/}(2), 154--158.

\bibitem[\protect\citeauthoryear{Taraldsen and Lindqvist}{Taraldsen and
  Lindqvist}{2016}]{TaraldsenLindqvist16renyi}
Taraldsen, G. and B.~H. Lindqvist (2016).
\newblock {Conditional probability and improper priors}.
\newblock {\em Communications in Statistics\/}~{\em 45\/}(17), 5007--5016.

\bibitem[\protect\citeauthoryear{Tufto, Lande, Ringsby, Engen, S{\ae}ther,
  Walla, and DeVries}{Tufto et~al.}{2012}]{Tufto2012butterflies}
Tufto, J., R.~Lande, T.-H. Ringsby, S.~Engen, B.-E. S{\ae}ther, T.~R. Walla,
  and P.~J. DeVries (2012).
\newblock {Estimating Brownian motion dispersal rate, longevity and population
  density from spatially explicit mark-recapture data on tropical butterflies}.
\newblock {\em Journal of Animal Ecology\/}~{\em 81}, 756--769.

\end{thebibliography}

\end{document}